%% file: paper.tex
\documentclass[12pt,twoside,reqno]{amsart}

\usepackage{bm}

\usepackage{amssymb,amsmath,amstext,amsthm,amsfonts,amscd,xcolor}

\usepackage{dsfont}

\usepackage{amsthm, enumerate}

\usepackage[ansinew]{inputenc}
\usepackage{graphicx}
\usepackage[mathscr]{eucal}
\usepackage{hyperref}

\newcommand{\Supp}{S}
\newcommand{\R}{\mathbb{R}}

\newcommand{\C}{\mathbb{C}}
\newcommand{\N}{\mathbb{N}}
\newcommand{\Z}{\mathbb{Z}}

\newcommand{\T}{\mathbb{T}}
\newcommand{\SL}{{\rm SL}}

\newcommand{\nbar}{\xoverline{n}}

\theoremstyle{plain}
\newtheorem{theorem}{Theorem}[section]
\newtheorem{proposition}{Proposition}[section]

\newtheorem{lemma}[proposition]{Lemma}
\theoremstyle{definition}

\newtheorem*{theorem*}{Theorem}

\theoremstyle{definition}
\newtheorem{remark}{Remark}[section]
\newtheorem{example}[theorem]{Example}
\numberwithin{equation}{section}

\newcommand{\abs}[1]{\left| #1 \right|} 
\newcommand{\norm}[1]{\left\|#1\right\|} 
\newcommand{\normtwo}[1]{
{\left\vert\kern-0.25ex\left\vert\kern-0.25ex\left\vert #1
    \right\vert\kern-0.25ex\right\vert\kern-0.25ex\right\vert} }

\newcommand{\bigo}{\mathcal{O}}

\newcommand{\ep}{\epsilon}
 
\newcommand{\la}{\lambda}

\newcommand{\om}{\omega}

\makeatletter
\newsavebox\myboxA
\newsavebox\myboxB
\newlength\mylenA

\newcommand*\xoverline[2][0.75]{%
    \sbox{\myboxA}{$\m@th#2$}%
    \setbox\myboxB\null
    \ht\myboxB=\ht\myboxA%
    \dp\myboxB=\dp\myboxA%
    \wd\myboxB=#1\wd\myboxA
    \sbox\myboxB{$\m@th\overline{\copy\myboxB}$}
    \setlength\mylenA{\the\wd\myboxA}
    \addtolength\mylenA{-\the\wd\myboxB}%
    \ifdim\wd\myboxB<\wd\myboxA%
       \rlap{\hskip 0.5\mylenA\usebox\myboxB}{\usebox\myboxA}%
    \else
        \hskip -0.5\mylenA\rlap{\usebox\myboxA}{\hskip 0.5\mylenA\usebox\myboxB}%
    \fi}
\makeatother

\newcommand{\comp}{^{\complement}}

\newcommand{\U}{\mathscr{U}}

\newcommand{\ind}{\mathds{1}}

\newcommand{\B}{\mathscr{B}}

\newcommand{\dist}{{\rm dist}}

\newcommand{\Bscr}{\mathscr{B}}

\newcommand{\Bcal}{\mathcal{B}}

\newcommand{\X}{\mathcal{X}}

\newcommand\restr[2]{{
  \left.\kern-\nulldelimiterspace 
  #1 
  \vphantom{\big|} 
  \right|_{#2} 
  }}

\newcommand{\Prob}{\mathrm{Prob}}
\newcommand{\Ascr}{\mathcal{A}}
\newcommand{\Gscr}{\mathcal{G}}

\newcommand{\supp}{\mathrm{supp} \, }
\newcommand{\Cscr}{C^0}
\newcommand{\Qop}{\mathcal{Q}}

\newcommand{\Av}{\mathrm{Av}}

\newcommand{\alfa}{\mathfrak{a}}
\newcommand{\muh}{{\nu}}

\newcommand{\Xplus}{{X^+}}

\newcommand{\Sigmah}{{\Omega}}

\newcommand{\f}{\bm{f}}

\newcommand{\Lip}{\mathrm{Lip}}

\title[Mixed Random-Quasiperiodic Cocycles]{Mixed Random-Quasiperiodic Cocycles}

\date{}

\begin{document}

\author[A. Cai]{Ao Cai}
\address{Departamento de Matem\'atica and CMAFcIO\\
Faculdade de Ci\^encias\\
Universidade de Lisboa\\
Portugal
}
\email{acai@ptmat.fc.ul.pt}

\author[P. Duarte]{Pedro Duarte}
\address{Departamento de Matem\'atica and CMAFcIO\\
Faculdade de Ci\^encias\\
Universidade de Lisboa\\
Portugal
}
\email{pduarte@ptmat.fc.ul.pt}

\author[S. Klein]{Silvius Klein}
\address{Departamento de Matem\'atica, Pontif\'icia Universidade Cat\'olica do Rio de Janeiro (PUC-Rio), Brazil}
\email{silviusk@mat.puc-rio.br}

\begin{abstract}
We introduce the concept of mixed random - quasiperiodic linear cocycles. We characterize the ergodicity of the base dynamics and establish a large deviations type estimate for certain types of observables. For the fiber dynamics we prove the uniform upper semicontinuity of the maximal Lyapunov exponent. This paper is meant to introduce a model to be studied in depth in further projects.
\end{abstract}

\maketitle


\section{Introduction}\label{intro}
\input{intro.tex}

\section{The base dynamics}\label{base}
\input{basedyn.tex}

\section{The fiber dynamics}\label{fiber}
\input{fiberdyn.tex}

\section{A motivating example and future work}\label{roadmap}

\input{roadmap.tex}

\medskip

\subsection*{Acknowledgments} The second author was supported  by Funda\c{c}\~{a}o para a Ci\^{e}ncia e a Tecnologia, under the projects: UID/MAT/04561/2013 and   PTDC/MAT-PUR/29126/2017.
The third author has been supported by the CNPq research grants 306369/2017-6 and 313777/2020-9.
 
\bigskip


\providecommand{\bysame}{\leavevmode\hbox to3em{\hrulefill}\thinspace}
\providecommand{\MR}{\relax\ifhmode\unskip\space\fi MR }
\providecommand{\MRhref}[2]{%
  \href{http://www.ams.org/mathscinet-getitem?mr=#1}{#2}
}
\providecommand{\href}[2]{#2}

\end{document}

%% file: intro.tex
\medskip

Consider a compactly supported probability measure $\muh$ on the group $\SL_m (\R)$ of $m$ by $m$ matrices with determinant $1$ and let
$$\Pi_n = g_{n-1} \ldots g_1 \, g_0 \, $$
be the random multiplicative process driven by this measure, where $\{g_n\}_{n\in\Z}$ is an i.i.d. sequence of $\SL_m (\R)$ valued random variables with common law $\muh$. 

By Furstenberg-Kesten's theorem (see~\cite{FK}), the geometric average 
$\frac{1}{n} \log \norm{\Pi_n}$ converges $\muh$-a.s. to a constant $L_1 (\muh)$ called the maximal Lyapunov exponent of the process.

An important example of such a process comes from the study of the Anderson model, the discrete random Schr\"odinger operator on $\ell^2 (\Z)$ used in solid state physics to model one dimensional disordered systems (e.g. semiconductors with impurities). This operator is given by
$$(H \psi)_n := - \psi_{n+1} - \psi_{n-1} + w_n \psi_n \qquad \forall n \in \Z \, ,$$
where $\psi = \{\psi_n\}_{n\in\Z}  \in \ell^2 (\Z)$ and $\{w_n\}_{n\in\Z}$ is an i.i.d. sequence of real valued random variables. 
The corresponding Schr\"odinger (or eigenvalue) equation $H \psi = E \psi$ is equivalent to
$$\begin{pmatrix} \psi_{n+1} \\ \psi_n \end{pmatrix} =
\begin{pmatrix} w_n-E & -1 \\ 1 & 0 \end{pmatrix} \, 
\begin{pmatrix} \psi_{n} \\ \psi_{n-1} \end{pmatrix} \quad \forall n\in\Z \, .
$$
Thus it can be solved by means of transfer matrices $\Pi_n = g_{n-1} \ldots g_1 \, g_0$, where
$g_n := \begin{pmatrix} w_n-E & -1 \\ 1 & 0 \end{pmatrix} $
are i.i.d. random matrices.

The behavior of the corresponding Lyapunov exponent as a function of the energy $E$ (e.g. its positivity or its continuity) 
is directly pertinent to the study of the spectral properties of the Schr\"odinger operator.

\medskip

At the other end of the range of ergodic comportment lies the discrete quasiperiodic Schr\"odinger operator, which in solid state physics is employed in the description of two dimensional crystal layers immersed in a magnetic field. This operator is given by
$$(H (\theta) \, \psi)_n := - \psi_{n+1} - \psi_{n-1} + v_n (\theta) \,  \psi_n \qquad \forall n \in \Z \, ,$$
where $v_n (\theta) = v (\theta + n \alpha)$ for some continuous function $v$ on the torus $\T^d = (\R/\Z)^d$, rationally independent frequency $\alpha \in \T^d$ and phase $\theta \in \T^d$.
The corresponding Schr\"odinger equation $H (\theta) \, \psi = E \, \psi$ gives rise to the  deterministic (quasiperiodic) multiplicative process
$$\Pi_n = \begin{pmatrix} v_{n-1} (\theta)  -E & -1 \\ 1 & 0 \end{pmatrix} \ldots 
\begin{pmatrix} v_1 (\theta) -E & -1 \\ 1 & 0 \end{pmatrix} \, \begin{pmatrix} v_0 (\theta) -E & -1 \\ 1 & 0 \end{pmatrix}
\, .$$

Both of these types of multiplicative processes can be studied in the more general framework of linear cocycles. A linear cocycle over an ergodic system $(X, f, \rho)$ (referred to as the base dynamics) is a skew product map of the form
$$ X \times \R^m \ni (x, u) \mapsto F (x, u) = \left( f (x), A(x) u \right) \in X \times \R^m ,$$
where $A \colon X \to \SL_m (\R)$ is a measurable function (referred to as the fiber map). The iterates of $F$ are 
$F^n (x, u) = \left( f^n (x), A^n (x) u \right)$, where $A^n$ is the multiplicative process
$$A^n (x) = A (f^{n-1} x) \ldots A (f (x) ) \, A (x) \, .$$
Its maximal Lyapunov exponent is defined as before, by Furstenberg-Kesten's theorem, as the $\rho$-a.e. limit of
$\frac{1}{n} \, \log \norm{ A^n (x) } $. 

When the base dynamics is a Bernoulli shift on a space of sequences and the fiber map $A$ depends only on the zeroth coordinate of the sequence, its iterates $A^n$ encode a random multiplicative system. When the base dynamics is a torus translation, the iterates of the fiber map define a quasiperiodic multiplicative process.

\medskip

In this paper we introduce the notion of a mixed random-quasiperiodic multiplicative process. An important example of such system  is related to the study of the discrete Schr\"odinger operator with mixed random-quasiperiodic potential
$$(H \psi)_n = - \psi_{n+1} - \psi_{n-1} + \left( v (\theta + n \alpha) + w_n \right) \, \psi_n \quad \forall n \in \Z .$$

The  random part of the potential, given by the sequence $\{w_n\}_{n\in\Z}$ of i.i.d. random variables, may be regarded as a perturbation of the quasiperiodic part $\{v (\theta+n\alpha)\}_{n\in\Z}$. A natural question is then to understand the influence of this random noise on the behavior of the system. For instance, is the Lyapunov exponent of the quasiperiodic system stable under such random perturbations? This question was posed to us by Jiangong You and it motivated this and several subsequent projects.

\medskip

A quasiperiodic cocycle can be identified with a pair $(\alpha, A)$, where  $\alpha \in \T^d$ is an ergodic  frequency (which defines the base dynamics, a torus translation) and $A \colon \T^d \to \SL_m (\R)$ is a continuous function (which induces the fiber action). Let $\Gscr$ be the (metric) space of quasiperiodic cocycles. It turns out that $\Gscr$ has a natural group structure, and in fact $(\Gscr, \circ)$ is a topological group.  Given a compactly supported measure $\muh$ on $\Gscr$ and an i.i.d. sequence $\{\om_n\}_{n\in\Z}$ of random variables with values in the group $\Gscr$ and with common law $\muh$, we may interpret the random product of quasiperiodic cocycles
$$\Pi_n = \om_{n-1} \circ \ldots \circ \om_1 \circ \om_0$$
as a mixed random-quasiperiodic multiplicative process. 

Another (not completely equivalent) way of defining such a process is to regard it as the iterates of a certain kind of linear cocycle over a mixed base dynamics. The latter is a skew-product of a Bernoulli shift with a random translation. 

\medskip

This paper is the first in a series of works regarding such mixed processes. Its purpose is to introduce the main concepts and to establish some (technical) results, to be used later, a common theme thereof being a certain uniform behavior in the quasiperiodic variable $\theta$ (which is natural, given the unique ergodicity of the torus translation). 

The paper is organized as follows. In Section~\ref{base} we define the mixed random-quasiperiodic base dynamics, characterize its ergodicity (Theorem~\ref{ergodic charact}) and establish a large deviations type estimate for certain observables (Theorem~\ref{base ldt}). In Section~\ref{fiber} we formally introduce the concept of mixed random-quasiperiodic cocycle driven by a measure on the group of quasiperiodic cocycles and establish a uniform upper large deviations type estimate (Theorem~\ref{fiber upper ldt}). As a consequence, we prove that the maximal Lyapunov exponent is upper semicontinuous as a function of the measure, relative to the Wasserstein distance. 

In Section~\ref{roadmap} we outline some of the upcoming works on the models introduced here, leading up to the stability under random noise of the Lyapunov exponent of a quasiperiodic cocycle. The second and third authors are grateful to Jiangong You for posing this question, that proved very fruitful, and to Nanjing University for their hospitality during an event in 2018 where the conversation took place.

%% file: basedyn.tex
Let $(\Sigmah, \Bcal)$ be a standard Borel space. That is, $\Sigmah$ is a Polish space (a separable, completely metrizable topological space) and $\Bcal$ is its Borel $\sigma$-algebra. 

Let $\muh \in \Prob_c (\Sigmah)$ be a compactly supported Borel probability measure on $\Sigmah$. Regarding $(\Sigmah, \muh)$ as a space of symbols, we consider the corresponding (invertible) Bernoulli system $\left(X, \sigma, \muh^\Z \right)$, where $X := \Sigmah^\Z$ and $\sigma \colon X \to X$ is the (invertible) Bernoulli shift: for $\om = \{ \om_n \}_{n\in\Z}  \in X$,  
$\sigma \om  :=  \{ \om_{n+1} \}_{n\in\Z}$. Consider also its non invertible factor on $X^+ := \Sigmah^\N$.

Let $\T^d = \left(\R/\Z \right)^d$ be the torus of dimension $d$, and denote by $m$ the Haar measure on its Borel $\sigma$-algebra.
 
Given a continuous function $\alfa \colon \Sigmah \to \T^d$, the skew-product map
\begin{equation}\label{base map}
\f \colon X \times \T^d \to X  \times \T^d \, , \quad \f (\om, \theta) := \left( \sigma \om, \theta + \alfa ( \om_0) \right)
\end{equation}
will be referred to as a {\em mixed random-quasiperiodic} (base) dynamics. 

This map preserves the measure $\muh^\Z\times m$ and it is the natural extension of the non invertible map on
$X^+ \times \T^d$ which preserves the measure $\muh^\N\times m$ and is defined by the same expression.

We will study the ergodicity of the mixed random-quasiperiodic system $\left( X \times \T^d, \f, \muh^\Z\times m \right)$. For simplicity, when this holds, we sometimes call the measure $\muh$ ergodic, or ergodic with respect to $\f$. 

We first consider a factor of this system, induced by the function $\alfa \colon \Sigmah \to \T^d$. Regard $\Sigma := \T^d$ as a space of symbols equipped with the push-forward measure $\mu := \alfa_\ast \muh$ and consider the skew-product map
\begin{equation}\label{base map 2}
f \colon \Sigma^\Z \times \T^d \to \Sigma^\Z  \times \T^d \, , \quad f (\beta, \theta) := \left( \sigma \beta, \theta + \beta_0 \right) ,
\end{equation}
where here $\sigma $ stands  for the Bernoulli shift on the space $\Sigma^\Z$ of sequences $\beta =  \{ \beta_{n} \}_{n\in\Z}$. The function
$$\pi \colon \Sigmah^\Z \times \T^d \to \Sigma^\Z \times \T^d, \quad 
\pi \left( \{ \om_n \}_n, \theta \right) = \left( \{ \alfa (\om_n) \}_n, \theta \right) $$
semi-conjugates  $\left(\Sigmah^\Z \times \T^d, \f,  \muh^\Z\times m \right)$ to $\left( \Sigma^\Z \times \T^d, f,  \mu^\Z\times m \right)$. Thus the second system is a factor of the first, showing in particular that the ergodicity of $\muh$ implies that of $\mu$. While in general the reverse implication is not true, in our case it does hold. That is because the action in the first coordinate is a Bernoulli shift which is mixing.

\begin{proposition}\label{prop ergodic}
The measure preserving dynamical system $\left(\f,  \muh^\Z\times m \right)$ is ergodic if and only if $\left(f,  \mu^\Z\times m \right)$ is ergodic.
\end{proposition}
\begin{proof}
It is enough to prove the reverse statement.

Recall that a measure preserving dynamical system $(\mathcal{X},f,\la)$ is ergodic if and only if for any $\varphi$, $\psi\in L^{2}(\mathcal{X})$,
\begin{equation}\label{ergo}
\lim_{n\to \infty}\frac{1}{n}\sum_{j=0}^{n-1}\int(\varphi\circ f^j)\psi \, d\la=\int \varphi \, d\la \, \int \psi \, d\la.
\end{equation}

This equivalent definition of ergodicity will allow us to make use of the mixing property of the Bernoulli shift.

Since $(\ref{ergo})$ is linear in $\varphi$ and $\psi$, in order to prove the ergodicity of $(\mathcal{X},f,\la)$ it is enough to find a subset $V\subset L^2 (\mathcal{X})$ such that ${\rm LS} (V)$, the linear span of $V$, is dense in $L^2(\mathcal{X})$ and
$(\ref{ergo})$ holds for any $\varphi, \psi \in V$.

We construct such a set $V \subset L^2 (\Sigmah^\Z \times \T^d)$ as an increasing limit of sets $V_n$ of functions depending on a finite number of variables.

Given  $\varphi \colon \Sigmah^\Z \to \R$ and $\psi \colon \T^d \to\R$, denote by $\varphi \otimes \psi$ the function on 
$ \Sigmah^\Z \times \T^d$ defined by $ \varphi \otimes \psi (\om, \theta) := \varphi (\om) \,  \psi (\theta)$.

Then for all $n\in \N$, let
$$V_n :=\left\{\varphi \otimes \psi \colon \varphi\in C^0_n(\Sigmah^\Z), \psi\in C^0(\T^d) \right\}  ,$$ 
where $C^0_n(\Sigmah^\Z)$ consists of all observables on $\Sigmah^\Z$ which depend only on the coordinates $(\omega_{-n},\cdots,\omega_0,\cdots,\omega_n)$. These observables are simply conditional expectations of  absolutely continuous functions with respect to the sub-algebra generated by the centered cylinder of length $2n+1$. 
The sequence of sets $\{V_n\}_{n\ge1}$ is clearly increasing, so let $V :=\bigcup_{n=0}^{\infty} V_n$.

Recall that the measure $\mu \in \Prob (\Sigma)$ is the push-forward of $\muh\in \Prob (\Sigmah)$ via the map $\alfa:\Sigmah\to \Sigma$. We may then consider the disintegration of $\muh$ into $\left\{\muh_\beta\right\}_{\beta \in \Sigma} \subset \Prob_c (\Sigmah)$ such that  $\muh=\int_\Sigma \muh_\beta d\mu(\beta)$. A direct computation yields that
$$
\muh^\Z=\int_{\Sigma^\Z} ( \prod_{i\in\Z}\muh_{\beta_i} )  \, d\mu^\Z(\{\beta_i\}_{i\in \Z}).
$$

We define $\Av \colon C^0_n(\Sigmah^\Z) \to C^0_n(\Sigma^\Z)$ by
$$
(\Av \,\varphi)(\{\beta_i\}_{i\in\Z}) := \int_{\Sigmah^\Z}  \varphi(\{\om_i\}_{i\in\Z}) \, d (\prod_{i\in\Z}\muh_{\beta_i})(\{\om_i\}_{i\in\Z}).
$$

Note that 
$$\int \Av \,\varphi \, d \mu^\Z = \int \varphi \, d \muh^\Z .$$

It is straightforward to check that for $N>2n+1$ and $\varphi,\phi \in C^0_n(\Sigmah^\Z)$ we have $\Av[(\varphi\circ \sigma^N)\phi]=\Av(\varphi\circ \sigma^N)\Av\,\phi$ and $\Av(\varphi \circ \sigma)=(\Av\,\varphi)\circ \sigma$ (we use the same symbol $\sigma$ to denote both the shift on $\Sigmah^\Z$ and on $\Sigma^\Z$).

Take $\varphi_1\in C^0_n(\Sigmah^\Z), \varphi_2\in C^0(\T^d)$ and $\varphi=\varphi_1 \otimes \varphi_2$. Similarly, take $\psi_1\in C^0_n(\Sigmah^\Z), \psi_2\in C^0(\T^d)$ and $\psi=\psi_1 \otimes \psi_2$.

For $N>2n+1$, we have
\begin{align*}
&\int \int (\varphi\circ \f^N) \, \psi \,d\muh^\Z\times m\\
=& \int \int \varphi_1(\sigma^N\omega)\varphi_2(\theta+\alfa(\omega_0) + \cdots + \alfa(\omega_{n-1})) \, \psi_1(\omega)\psi_2(\theta)d\muh^\Z(\omega) d m(\theta)\\
=& \int \int \varphi_1(\sigma^N\omega)\psi_1(\omega) \, \varphi_2(\theta+\alfa(\omega_0) + \cdots + \alfa(\omega_{n-1}))\psi_2(\theta)d\muh^\Z(\omega) d m(\theta)\\
=& \int \int \Av[(\varphi_1\circ \sigma^N)\psi_1] \, \varphi_2(\theta+\beta_0+\cdots+\beta_{n-1})\psi_2(\theta)d\mu^{\Z}         (\{\beta_i\}_{i\in\Z})dm(\theta)\\
=& \int \int [(\Av\,\varphi_1)\circ \sigma^N](\Av \,\psi_1) \, \varphi_2(\theta+\beta_0+\cdots+\beta_{n-1})\psi_2(\theta)d\mu^{\Z} dm(\theta)\\
=& \int \int [(\Av\, \varphi_1) \otimes \varphi_2]\circ f^N \cdot [(\Av \,\psi_1) \otimes \psi_2] \, d\mu^\Z\times m,
\end{align*}
which converges in the Ces\`{a}ro sense to
\begin{equation}\label{cesaro}
\int (\Av\,\varphi_1)\otimes \varphi_2 \,d\mu^\Z\times m \, \cdot \,  \int(\Av\,\psi_1)\otimes \psi_2 \,d\mu^\Z\times m
\end{equation}
since $\left(f,  \mu^\Z\times m \right)$ is ergodic. 

Moreover, we have (same computations for $\psi$)
\begin{align*}
&\int (\Av\,\varphi_1)\otimes \varphi_2 \,d\mu^\Z\times m \\
=&\int \Av\,\varphi_1 \,d\mu^\Z \cdot \int \varphi_2 \,dm \\
=& \int \varphi_1 \,d\muh^\Z \cdot \int \varphi_2 \,dm\\
=& \int \varphi_1\otimes \varphi_2 \,d\muh^\Z\times m\\
=& \int \varphi \,d\muh^\Z\times m.
\end{align*}
Therefore $\eqref{cesaro}$ is equal to $\displaystyle \int \varphi \,d\muh^\Z\times m \cdot \int \psi\, d\muh^\Z\times m$ and we conclude that 
$$
\int (\varphi\circ f^N)\psi \,d\muh^\Z\times m  \xrightarrow[\text{$N \rightarrow \infty$}]{\text{Ces\`{a}ro}}\int \varphi \,d\muh^\Z\times m \cdot \int \psi\, d\muh^\Z\times m
$$
holds for any $\varphi,\psi\in V_n$ and $n\in \N$ thus  for any $\varphi,\psi\in V$.

This proves that $\left(\f,  \muh^\Z\times m \right)$ is also ergodic.
\end{proof}

Let us consider two basic examples of mixed random-quasiperiodic transformations as in~\eqref{base map}.

\begin{example}
Given a standard Borel probability measure space $(\Sigmah, \muh)$ and a frequency $\alpha \in \T^d$, let $\alfa \colon \Sigmah \to \T^d$ be the constant function $\alfa (\om_0) \equiv \alpha$. Then the corresponding skew-product map $\f$ on $\Sigmah^\Z \times \T^d$ is given by 
$$\f (\om, \theta) = (\sigma \om, \theta + \alpha) \, .$$

Thus the system $(\f, \muh^\Z \times m)$ is just the product between the Bernoulli shift $\sigma$ and the torus translation by $\alpha$, which we denote by $\tau_\alpha$. 

Moreover, since $\mu = \alfa_\ast \muh = \delta_\alpha$ (the Dirac measure centered at $\alpha$), 
its factor $(f, \mu^\Z \times m)$ as defined above is clearly isomorphic to the torus translation $(\tau_\alpha, m)$. By  Proposition~\ref{prop ergodic}, $(\f, \muh^\Z \times m)$ is ergodic if and only if $(\tau_\alpha, m)$ is ergodic, which is of course well known.
\end{example}

\begin{example}
Given a standard Borel probability measure space $(S, \rho)$ and $\mu \in \Prob (\T^d)$, let $\Sigmah := \T^d \times S$, $\muh := \mu \times \rho$ and let $\alfa \colon \Sigmah \to \T^d$ be the projection in the first coordinate, $\alfa (\beta, b) = \beta$. It clearly holds that $\alfa_\ast \muh = \alfa_\ast (\mu \times \rho) = \mu$.

The corresponding skew-product map on  $\Sigmah^\Z \times \T^d$ is given by
$$\f (\{\om_n\}, \theta) = (\{\om_{n+1}\}, \theta + \alfa (\om_0)) \, ,$$
while its factor on $\Sigma^\Z \times \T^d$ is
$$f (\{\beta_n\}, \theta) = (\{\beta_{n+1}\}, \theta + \beta_0 ) \, .$$

By  Proposition~\ref{prop ergodic}, $\muh$ is ergodic with respect to $\f$ if and only if $\mu$ is ergodic with respect to $f$. \end{example}

\subsection{Stochastic dynamical systems} We introduce some general concepts that will be used throughout the paper.

Given a metric space $(M, d)$, denote by $C^0 (M)$, $C_b (M)$, $\Lip (M)$, respectively, the spaces of continuous functions, continuous and bounded functions and Lipschitz continuous real valued functions on $M$. Let $\norm{g}_0$ denote the uniform norm of a function $g \in C_b (M)$ and let $\norm{g}_\Lip$ denote the best 
 Lipschitz constant of a function $g \in \Lip (M)$. 
 
The following Urysohn type lemma will be needed in the sequel.

\begin{lemma} \label{Urysohn}
Let $(M, d)$ be a metric space and let $\nu$ be a Borel probability measure in $M$. Given a closed set $L \subset M$ and $\ep > 0$ there are an open set $D \supset L$ such that $\nu (D) < \nu (L) + \ep$ and a Lipschitz continuous function $g \colon M \to [0, 1]$ such that $\ind_L \le g \le \ind_D$. 
\end{lemma}

\begin{proof}
For every $\delta > 0$ let $L_\delta := \left\{ x \in M \colon d (x, L) < \delta \right\}$ be the open $\delta$-neighborhood of $L$. Since $L$ is closed we have that $\bigcap_{\delta > 0} \, L_\delta = L$. Then $\nu (L_\delta) \to \nu (L)$ as $\delta \to 0$, so there is $\delta_0 = \delta_0 (L, \ep, \nu) > 0$ such that $\nu (L_{\delta_0}) < \nu (L) + \ep$.  

Let $D :=  L_{\delta_0}$ and note that $d \left( L, D\comp \right) = d \left( L, L_{\delta_0}\comp \right) \ge \delta_0 > 0$. One can then easily verify that the function $g \colon M \to \R$, 
$$g (x) :=  \frac{d (x, D\comp)}{d (x, D\comp) + d (x, L)} $$
is Lipschitz continuous with $\norm{g}_\Lip \le \frac{1}{\delta_0}$, while clearly $\ind_L \le g \le \ind_D$.

The main point here is that the closed set $L$ need not be compact, as its distance to the closed set $D\comp$ is already bounded away from zero.
\end{proof}

Let $\Prob (M)$ denote the space of Borel probability measures on $M$ and define the weak* convergence of a sequence $\nu_n \to \nu$ by $\int \phi d \nu_n \to \int \phi d \nu$ for all $\phi \in C_b (M)$. 

Furthermore, let
$$\Prob_1 (M) := \{ \nu \in \Prob (M) \colon \int_M d (x, x_0) \, d \nu (x) < \infty \} \, ,$$
where $x_0 \in M$ is an arbitrary point (whose choice is of course inconsequential). 
Note that $\Prob_c (M)$, the space of compactly supported Borel probability measures on $M$, is contained in $\Prob_1 (M)$.

 If $M$ is a compact metric space then the weak* convergence defines the weak topology on $\Prob_c (M) = \Prob_1 (M) = \Prob (M)$ and this topology is  compact and metrizable. If $M$ is a (more general) Polish metric space (thus not necessarily compact), the weak topology on $\Prob_1 (M)$ is defined by the weak* convergence $\nu_n \to \nu$ together with the convergence $\int d (x, x_0) \,  d \nu_n \to \int d (x, x_0) \,  d \nu$ for some (and hence all) $x_0 \in M$. Then $\Prob_1 (M)$ is itself a Polish space.

In either case, consider the Wasserstein (or Kantorovich-Rubinstein) distance $W_1$ in the space $\Prob_1 (M)$, where
$$W_1 (\nu, \nu') := \sup \left\{  \int g \, d (\nu-\nu')  \colon g \in \Lip (M), \norm{g}_\Lip \le 1 \right\} \, .$$

It is well known that this distance metrizes the weak topology on $\Prob_1 (M)$, see~\cite[Chapter I.6]{Villani} for this and all other related concepts mentioned above.

\medskip

A {\em stochastic dynamical system} (SDS) on $M$ (also called a random walk in~\cite{Furstenberg}) is any continuous map  $K \colon M\to\Prob (M)$, $x\mapsto K_x$.
An SDS $K$ on $M$ induces a bounded linear operator (called the Markov operator)
$\Qop_K \colon C_b (M)\to C_b (M)$ defined by
$$ (\Qop_K \varphi)(x):= \int_M \varphi(y)\, dK_x(y) .$$
It also induces the adjoint operator
$\Qop_K^\ast :\Prob (M)\to \Prob (M)$ of $\Qop_K$ characterized by
$$ \Qop_K^\ast \nu = K\ast \nu  := \int_M  K_x\, d\nu(x) . $$
A measure $\nu\in \Prob(M)$ is called \textit{$K$-stationary} if
$\Qop_K^\ast \nu=\nu$. We denote by $\Prob_K(M)$ the convex and compact subspace of
all $K$-stationary probability measures on $M$.

Let $(G, \cdot)$ be a topological group acting on $M$ from the left. Denote by $\tau_g \colon M \to M$ the action on $M$ by $g\in G$, that is, $\tau_g (x) = g x$.

Given $\mu \in \Prob_c (G)$ and $\nu \in \Prob_c (M)$, the convolution $\mu \ast \nu \in \Prob_c (M)$ is given by
$$\mu \ast \nu (E) := \int_M \int_G \ind_E (g x) \, d \mu (g) d \nu (x) $$
for any Borel set $E \subset M$.

Then
$$\mu \ast \nu = \int_G \left( \tau_g \right)_\ast \nu \, d \mu (g) \, ,$$
where $\left( \tau_g \right)_\ast \nu $ is the push-forward probability measure
$$\left( \tau_g \right)_\ast \nu (E) := \nu \left(\tau_g^{-1} E \right) = \nu \left (g^{-1} E \right) \, .$$

A probability measure $\mu \in \Prob_c (G)$ determines an SDS on $M$ by
$$M \ni x \mapsto \mu \ast \delta_x = \int_G \delta_{g x } \, d \mu (g) \in \Prob _c (M) \, .$$

The associated Markov operator $\Qop_\mu \colon C_b (M) \to C_b (M)$ is given by
$$\left( \Qop_\mu \phi \right) (x) = \int_M \phi (y) \, d \mu \ast \delta_x (y) = \int_G \phi (g x) \, d \mu (g) \, .$$

Moreover, its dual operator $\Qop_\mu^\ast \colon \Prob_c (M) \to \Prob_c (M)$ is
$$\Qop_\mu^\ast \nu = \int_M \mu \ast \delta_x \, d \nu (x) = \mu \ast \nu \, .$$

Let
$$\Prob_\mu (M) := \left\{ \nu \in \Prob_c (M) \colon  \mu \ast \nu = \nu \right\} $$
be the set of $\mu$-stationary measures on $M$, that is, the fixed points of the dual Markov operator $\Qop_\mu^\ast$.

Given such a $\mu$-stationary measure $\nu$, any observable $\phi \colon M \to \R$ for which
$$\left( \Qop_\mu \phi \right) (x) = \phi (x) \quad \text{ for } \nu \text{ a.e. } x \in M$$
is called a $\nu$-stationary observable.

\medskip

Specializing to $G = M = \T^d$ seen as an additive group, for $\alpha, \theta \in \T^d$ and $\mu \in \Prob (\T^d)$ we have
$\tau_\alpha (\theta) = \theta + \alpha$,
$$\left( \Qop_\mu \phi \right) (\theta) = \int_{\T^d} \phi (\theta + \alpha) \, d \mu (\alpha) $$
and
$$ \Qop_\mu^\ast \nu (E) =    \int_{\T^d} \left( \tau_\alpha \right)_\ast \nu (E) \, d \mu (\alpha) = 
\int_{\T^d} \nu \left( \tau_\alpha^{-1} E \right)  \, d \mu (\alpha) $$
for any Borel measurable set $E \subset \T^d$.

Let $m$ be the Haar measure on $\T^d$. Note that $m$ is $\mu$-stationary (since it is translation invariant). 

Finally, given any $k\in\Z^d$, we define the corresponding  Fourier coefficient of the measure
$\mu$ by
$$ \hat\mu(k) := \int_{\Sigma} e^{2\pi i \langle k, \alpha\rangle}\, d\mu(\alpha) .$$

\subsection{Ergodicity of the base dynamics}
Proposition~\ref{prop ergodic} reduces the study of the ergodicity of a skew product map like~\eqref{base map} to that of its factor~\eqref{base map 2}. We will then study the latter.

The following result provides various characterizations of the ergodicity of the base transformation. Some of them, e.g. (4) and (5) can also be deduced from Anzai's theorem (see~\cite[Theorem 4.8]{Petersen}) on the ergodicity of general skew products.

\begin{theorem}
\label{ergodic charact}
Let $\mu \in \Prob (\Sigma)$ where $\Sigma = \T^d$, and consider the skew product map on $\Sigma^\Z \times \T^d$ given by $f ( \{\beta_i\}, \theta) = ( \sigma \{\beta_i\}, \theta + \beta_0 )$. 
The following statements are equivalent:
\begin{enumerate}
	\item[(1)] $f$ is ergodic w.r.t. $\mu^\Z\times m$;
	\item[(2)] $f$ is ergodic w.r.t. $\mu^\N\times m$;
	\item[(3)] Every $m$-stationary observable $\varphi\in L^\infty(\T^d)$ is constant   $m$-a.e.;
	\item[(4)] $\hat \mu(k)\neq 1$  for every $k\in\Z^d\setminus\{0\}$;
	\item[(5)] For every $k\in\Z^d\setminus\{0\}$ there exists $\alpha\in \Supp$ such that $\langle k, \alpha\rangle\notin \Z$;
	\item[(6)] $ \T^d=\overline{\cup_{n\geq 1} \Supp^n}$ where $\Supp = \supp (\mu)$ and
	$\Supp^n:=\Supp+\Supp^{n-1}$ \, $\forall n\geq 2$;
	\item[(7)]  $m$ is the unique $\mu$-stationary measure in $\Prob(\T^d)$,
	\item[(8)]  $\lim_{n\to +\infty}
	\frac{1}{n}\,\sum_{j=0}^{n-1}(\Qop_\mu^j \varphi)(\theta)=\int_{\T^d} \varphi\, dm$, \, $\forall\, \theta\in\T^d$ \, $\forall \varphi\in \Cscr(\T^d)$.
\end{enumerate}
\end{theorem}

\begin{proof}
$(1) \, \Rightarrow\, (2)$ holds trivially because $f$ in (2) is a factor $f$ in (1), i.e., because of the commutativity of the following diagram of measure preserving transformations.
$$
\begin{CD}
X    @>f>>  X\\
@V\pi VV        @VV\pi V\\
\Xplus     @>f>>  \Xplus
\end{CD}
$$

Conversely, $(2) \, \Rightarrow\, (1)$ holds
by Lemma 5.3.1 in~\cite{Hochman}.

The equivalence
$(2) \, \Leftrightarrow\, (3)$ follows from
Proposition 5.13 in~\cite{Viana-book}.

\bigskip

Given a bounded measurable function
$\varphi:\T^d\to \C$, we have  $\varphi\in L^2(\T^d, m)$. Consider its Fourier series
$$ \varphi =\sum_{k\in\Z^d} \hat\varphi(k) \, e_k  \quad \text{ with } \;
e_k(\theta):= e^{2\pi i \langle k, \theta\rangle}. $$
A simple calculation shows that
$$ \Qop_\mu \varphi = \sum_{k\in\Z^d}
\hat \mu(k)\, \hat\varphi(k) \, e_k .$$

$(3) \, \Rightarrow\, (4)$: If $\hat \mu(k)=1$
for some $k\in \Z^d\setminus\{0\}$, then $e_k$ is a non constant $m$-stationary observable.
In other words, if (4) fails then so does (3).

$(4) \, \Rightarrow\, (3)$:
Given $\varphi$  \, $m$-stationary, comparing the two Fourier developments above, for all $k\in\Z^d$ \;
$ \hat \mu(k)\, \hat\varphi(k) = \hat\varphi(k)$  \, $\Leftrightarrow$\,
$\hat\varphi(k)\, (\hat \mu(k)-1) = 0$.
By (4) we then get $\hat\varphi(k)=0$ for all $k\in\Z^d\setminus\{0\}$,
which implies that $\varphi=\hat \varphi(0)$ is $m$-a.e. constant.
This proves (3).

\bigskip

Since $\hat \mu(k)$ is an average of a continuous function with values
on the unit circle, we have
$$ \hat \mu(k)=1 \quad \Leftrightarrow \quad
e^{2\pi i \langle k, \alpha\rangle}= 1,\;
\forall \alpha \in \Supp \quad \Leftrightarrow \quad
\langle k, \alpha\rangle\in\Z,\; \forall \alpha \in \Supp .$$
This proves that  $(4) \, \Leftrightarrow\, (5)$.

\bigskip

$(5) \, \Rightarrow\, (6)$:
Let  $H=\overline{\cup_{n\geq 1} \Supp^n}$ and assume that
$H\neq \T^d$. By definition $H$ is a subsemigroup of $\T^d$.
By Poincar\'e recurrence theorem, $H$ is also a group.
By Pontryagin's duality for locally compact abelian groups,
there exists a non trivial character $e_k:\T^d\to\C$ which contains $H$ in its kernel. In particular this implies that
there exists $k\in\Z^d\setminus\{0\}$ such that
$\langle k, \beta \rangle\in\Z$ for all $\beta\in \Supp$.
This argument shows that  if (6) fails then so does (5).

$(6) \, \Rightarrow\, (5)$:
Assume that (5) does not hold, i.e., for some $k\in\Z^d\setminus\{0\}$
we have $\langle k, \alpha \rangle\in\Z$ for all $\alpha\in \Supp$. Then $e_k$ is a non trivial character
of $\T^d$ and $H:=\{ \theta\in\T^d\colon e_k(\theta)=1 \}$
is a proper sub-torus, i.e. a compact subgroup of $\T^d$.
The assumption implies that $\Supp\subset H$, and since $H$ is a group, $S^n\subset H$, $\forall n\geq 1$. This proves that
(6) fails.

\bigskip

Since the adjoint operator
$\Qop_\mu^\ast:\Prob(\T^d)\to \Prob(\T^d)$ satisfies
$\Qop_\mu^\ast \pi=\mu\ast \pi$, denoting by
$\mu^{\ast j}:= \mu \ast  \cdots \ast \mu$ the $j$-th convolution power of $\mu$,
we have $(\Qop_\mu^\ast)^n \delta_0= \mu^{\ast n}$\, $\forall n\in\N$.

\begin{lemma}\label{lem1.2}
Any sublimit of the sequence
$\pi_n:= \frac{1}{n}\,\sum_{j=0}^{n-1} \mu^{\ast j} $ is a $\mu$-stationary measure.
\end{lemma}

\begin{proof}
Given $\varphi\in\Cscr(\T^d)$,
\begin{align*}
\langle \Qop_\mu \varphi-\varphi, \pi_n \rangle &=
\frac{1}{n}\, \sum_{j=0}^{n-1} \langle \Qop_\mu \varphi-\varphi, (\Qop_\mu^\ast)^j\delta_0  \rangle\\
&=
\frac{1}{n}\, \sum_{j=0}^{n-1}  (\Qop_\mu^{j+1} \varphi)(0)-(\Qop_\mu^j\varphi)(0)\\
&= \frac{1}{n} \, ((\Qop_\mu^n \varphi)(0)-\varphi(0)) = \bigo (\frac{1}{n}) .
\end{align*}
Hence, if $\pi\in \Prob(\T^d)$ is a sublimit of $\pi_n$, taking the limit along the corresponding subsequence
of integers we have
$$
\langle \varphi, \Qop_\mu^\ast \pi-\pi  \rangle =  \langle \Qop_\mu \varphi-\varphi, \pi \rangle=  0 ,
$$
which implies that $\Qop_\mu^\ast\pi=\pi$.
\end{proof}

$(2) \, \Rightarrow\, (8)$:
By ergodicity of $f$ w.r.t. $\mu^\N\times m$
and Birkhoff Ergodic Theorem, given $\varphi\in\Cscr(\T^d)$ there exists a full measure set of $(\omega,\theta)\in \Supp^\N\times\T^d$
with
$$ \lim_{n\to +\infty} \frac{1}{n}\, \sum_{j=0}^{n-1} \varphi(\theta+\tau^j(\omega)) =\int \varphi\, dm ,$$
where $\tau^j(\omega)=\omega_0+\cdots+\omega_{j-1}$ and $\omega=\{\omega_j\}_{j\in \N}$. Hence there exists a Borel set $\Bscr\subset \T^d$
with $m(\Bscr)=1$ such that, applying the Dominated Convergence Theorem, we have for all $\theta\in\Bscr$,
$$ \lim_{n\to +\infty} \frac{1}{n}\, \sum_{j=0}^{n-1} (\Qop_\mu^j \varphi)(\theta) =\int \varphi\, dm .$$
The set $\Bscr$ depends on the continuous function $\varphi$, but since the space $\Cscr(\T^d)$ is separable
we can choose this Borel set $\Bscr$ so that the previous limit holds for every $\theta\in\Bscr$ and $\varphi\in\Cscr(\T^d)$. This implies the following weak* convergence in $\Prob(\T^d)$:
$$  \lim_{n\to +\infty} \frac{1}{n}\, \sum_{j=0}^{n-1} (\Qop_\mu^\ast)^j \delta_\theta = m  . $$
Given any $\theta'\notin\Bscr$ take $\theta\in \Bscr$. Convolving both sides on the right by $\delta_{\theta'-\theta}$ we get
$$  \lim_{n\to +\infty} \frac{1}{n}\, \sum_{j=0}^{n-1} (\Qop_\mu^\ast)^j \delta_{\theta'}
= \lim_{n\to +\infty} \frac{1}{n}\, \sum_{j=0}^{n-1} \mu^{\ast j}\ast \delta_\theta \ast \delta_{\theta'-\theta} = m\ast\delta_{\theta'-\theta} = m  , $$
 which proves (8).

$(8) \, \Rightarrow\, (7)$:
If there exists $\eta\neq m$ in $\Prob_\mu(\T^d)$, then there exists at least one more ergodic measure $\zeta\neq m$ such that $\zeta$ is an extreme point of $\Prob_\mu(\T^d)$.
Choosing $\varphi \in\Cscr(\T^d)$ such that $\int\varphi\, d\zeta \neq \int \varphi\, dm$,
by Birkhoff Ergodic Theorem there exists $\theta\in \T^d$ such that
$$ \lim_{n\to +\infty}
\frac{1}{n}\,\sum_{j=0}^{n-1}(\Qop_\mu^j \varphi)(\theta)=\int \varphi\, d\zeta  \neq \int\varphi\, dm. $$
which contradicts (8).

$(7) \, \Rightarrow\, (6)$:
Consider the compact subgroup $H:=\overline{\cup_{n\geq 1} \Supp^n}$.
If (6) fails then $H\neq \T^d$ and by Lemma \ref{lem1.2} we can construct a stationary
measure $\pi\in\Prob_\mu(\T^d)$ with $\supp(\pi)\subset H$.
This shows that $\pi\neq m$ and hence there is more than one stationary measure.
\end{proof}

\begin{proposition}\label{prop1.3}
If $f$ is ergodic w.r.t. $\mu^\Z\times m$ then
the convergence in item (8) of Proposition~\ref{ergodic charact}
holds uniformly in $\theta\in\T^d$.
\end{proposition}

\begin{proof}
We prove it by contradiction. Assume there $\exists \,\epsilon>0$, $\exists \,n_k\to \infty$ and $\exists\,\theta_k\in \T^d, k\in \N^+$ such that
$$
\abs{
	\frac{1}{n_k}\,\sum_{j=0}^{n_k-1}(\Qop_\mu^j \varphi)(\theta_k)-\int_{\T^d} \varphi\, dm } >\epsilon.
$$
Since $\T^d$ is compact, we can assume $\theta_k\to \theta$ for some $\theta\in \T^d$. Writing
\begin{align*}
& \frac{1}{n_k}\,\sum_{j=0}^{n_k-1}(\Qop_\mu^j \varphi)(\theta)-\int \varphi\, dm    = \\
 & \frac{1}{n_k}\,\sum_{j=0}^{n_k-1}(\Qop_\mu^j \varphi)(\theta)-\frac{1}{n_k}\,\sum_{j=0}^{n_k-1}(\Qop_\mu^j \varphi)(\theta_k)  + \frac{1}{n_k}\,\sum_{j=0}^{n_k-1}(\Qop_\mu^j \varphi)(\theta_k)-\int \varphi\, dm
 \end{align*}
 we have
 \begin{align*}
 & \abs{\frac{1}{n_k}\,\sum_{j=0}^{n_k-1}(\Qop_\mu^j \varphi)(\theta)-\int \varphi\, dm }  \ge \\
&  \abs{\frac{1}{n_k}\,\sum_{j=0}^{n_k-1}(\Qop_\mu^j \varphi)(\theta_k)-\int \varphi\, dm} - \\
& \abs{\frac{1}{n_k}\,\sum_{j=0}^{n_k-1}(\Qop_\mu^j \varphi)(\theta)-\frac{1}{n_k}\,\sum_{j=0}^{n_k-1}(\Qop_\mu^j \varphi)(\theta_k)}  \ge \,   \epsilon-\frac{\epsilon}{2} \, \geq  \, \frac{\epsilon}{2},
\end{align*}
where the second inequality is due to the definition of $\Qop_\mu^j \varphi$ and to the uniform continuity of $\varphi$ on $\T^d$. This contradicts (8).
\end{proof}

\subsection{Uniform convergence of the Birkhoff sums}
We return to the study our original base dynamics $\f \colon X\times \T^d\to X\times \T^d$ defined by~\eqref{base map}, where $X=\Sigmah^\Z$. Since the ergodicity of $\f$ is equivalent to that of its factor $f$, and since these two maps share similar expressions,  to simplify notations, from now on we let $f$ refer to either one of them.

Under the ergodicity assumption, we prove that for a full measure set of points $\omega\in X$, given any continuous observable $\phi \colon X\times\T^d\to \R$, the corresponding Birkhoff time averages converge to the space average uniformly in $\theta\in \T^d$.

\begin{lemma}
	\label{Birkhoff LDE1}
	Let $\muh \in \Prob_c (\Sigmah)$ and assume that $f$ is ergodic w.r.t. $\muh^\Z\times m$. There is a full measure set $X' \subset X$ such that 
	given any observable $\phi\in \Cscr(X\times \T^d)$,
	for all $\omega\in X'$ we have
	$$\lim_{n\to +\infty}
	\frac{1}{n}\,\sum_{j=0}^{n-1}\phi(f^j(\omega,\theta)) = \int \phi\, d(\muh^\Z\times m)$$
	with uniform convergence in $\theta\in\T^d$.
\end{lemma}

\begin{proof} 
Let $\Supp:= \supp \, \muh$ and $\X :=\Supp^\Z$. Since $\X$ is compact, $\sigma$-invariant and $\muh^Z (\Z)=1$, in order to prove the above convergence of the Birkhoff means for observables $\phi\in \Cscr(X\times \T^d)$ we may simply consider their restrictions to the compact, metrizable space $\X \times \T^d$.
As $\Cscr(\X\times \T^d)$ is separable, it admits a countable and dense subset $\{\phi_j \colon j\geq 1 \}$.

Denote by
$$ \Bscr_j:= \big\{ (\omega,\theta)\in \X\times \T^d \colon  \lim_{n\to +\infty} \frac{1}{n}\,\sum_{i=0}^{n-1} \phi_j(f^i(\omega,\theta))= \int \phi_j\, d(\muh^\Z\times m)
\big\}.$$ If we denote $\Bscr=\bigcap_{j\geq 1}\Bscr_j$, then by the Birkhoff Ergodic Theorem $(\muh^\Z\times m)(\Bscr)=1$. Thus for $m$-a.e. $\theta\in \T^d$, $\muh^\Z(\Bscr_{\theta})=1$ where $\Bscr_{\theta}=\left\{\omega\in \X \colon (\omega,\theta)\in \Bscr \right\}$.

Fix $\theta_0\in \T^d$ such that $\muh^\Z(\Bscr_{\theta_0})=1$. For $\muh^\Z$-a.e. $\omega$ (in fact for $\omega\in \Bscr_{\theta_0} \subset \X$), we have that for all $\phi\in \Cscr(X\times \T^d)$ and $\epsilon>0$, there exists $n_0$ such that for all $n\geq n_0$ and for $j$ large enough
\begin{align*}
&\abs{ \frac{1}{n}\,\sum_{i=0}^{n-1} \phi(f^i(\omega,\theta_0))- \int \phi\, d(\muh^\Z\times m)}\\
\leq &\abs{\frac{1}{n}\,\sum_{i=0}^{n-1} \phi(f^i(\omega,\theta_0))-\frac{1}{n}\,\sum_{i=0}^{n-1} \phi_j(f^i(\omega,\theta_0))}+\\
&\abs{\frac{1}{n}\,\sum_{i=0}^{n-1} \phi_j(f^i(\omega,\theta_0))-\int \phi_j\, d(\muh^\Z\times m)}+\\
&\abs{\int \phi_j\, d(\mu^\Z\times m)-\int \phi\, d(\muh^\Z\times m)}\\
\leq & \frac{\epsilon}{3}+\frac{\epsilon}{3}+\frac{\epsilon}{3} \,
\leq \,  \epsilon,
\end{align*}
where in the second inequality we used the density of $\{\phi_j \colon j\ge 1\}$ in  $\Cscr(\X\times \T^d)$ and the definition of $\Bscr$. This shows that for $\muh^\Z$-a.e. $\omega$, the following weak* convergence holds:
\begin{equation}\label{theta0}
\frac{1}{n}\,\sum_{i=0}^{n-1}\delta_{f^i(\omega,\theta_0)}\rightarrow \muh^\Z\times m .
\end{equation}

This in fact holds for any $\theta\in \T^d$ (not just for the given $\theta_0$) by the $m$-invariance of the torus translation.
Indeed, the action of $\T^d$ on $X \times \T^d$ given by $\theta \cdot (\om, \theta') = (\om, \theta + \theta')$ induces a convolution of measures and a direct computation shows that
$\delta_{\theta-\theta_0}\ast\delta_{f^j(\omega,\theta_0)}=\delta_{f^j(\omega,\theta)}$ for all $j\geq 1$ and $\delta_{\theta-\theta_0}\ast(\muh^{\Z}\times m)=\muh^{\Z}\times m$. Then by (\ref{theta0}) and the weak* continuity of the convolution operation, for $\muh^\Z$-a.e. $\omega$ and every $\theta\in \T^d$,
$$
\frac{1}{n}\,\sum_{i=0}^{n-1}\delta_{f^i(\omega,\theta)}\rightarrow \muh^\Z\times m.
$$

This is equivalent to saying that for $\muh^\Z$-a.e. $\omega$ (that is, for $\omega\in \Bscr_{\theta_0}$), for all $\phi\in \Cscr(X\times \T^d)$ and all $\theta\in\T^d$,
$$
\frac{1}{n}\,\sum_{j=0}^{n-1}\phi(f^j(\omega,\theta)) \to \int \phi\, d(\muh^\Z\times m).
$$

We prove the uniform convergence in $\theta$ by contradiction. Assume that there are $\omega\in \Bscr_{\theta_0}$, $\epsilon>0$, $n_k\to \infty$ and $\theta_k\in \T^d$ for all $k \ge 1$ such that
$$
\abs{\frac{1}{n_k}\,\sum_{j=0}^{n_k-1}\phi(f^j(\omega,\theta_k))-\int \phi\, d(\mu^\Z\times m)} \geq \epsilon.
$$

Since $\T^d$ is compact, by passing to a subsequence we may assume that $\theta_k\to\theta$. Then for $k$ sufficiently large we have:
\begin{align*}
&\abs{\frac{1}{n_k}\,\sum_{j=0}^{n_k-1}\phi(f^j(\omega,\theta))-\int \phi\, d(\muh^\Z\times m))}\\
\ge& \abs{\frac{1}{n_k}\,\sum_{j=0}^{n_k-1}\phi(f^j(\omega,\theta_k))-\int \phi\, d(\muh^\Z\times m))}\\
- & \abs{\frac{1}{n_k}\,\sum_{j=0}^{n_k-1}\phi(f^j(\omega,\theta))-\frac{1}{n_k}\,\sum_{j=0}^{n_k-1}\phi(f^j(\omega,\theta_k))}\\
\ge & \epsilon-\frac{\epsilon}{2}
\ge \frac{\epsilon}{2}.
\end{align*}

The second inequality follows from the fact that, for $k$ large enough, 
$$\abs{\phi (\om', \theta_k) -  \phi (\om', \theta)} < \frac{\ep}{2} \quad \forall \om' \in \X ,$$ 
which is due to the uniform continuity of $\phi$ on the compact set $\X \times \T^d$.  This contradicts the pointwise convergence for $\theta$.
\end{proof}

The next result establishes a large deviations type estimate over ergodic mixed random-quasiperiodic systems, for continuous observables that depend on finitely many coordinates. The estimate is uniform in the quasiperiodic variable $\theta$ and also in the measure determining the random variable.

\begin{theorem}\label{base ldt}
Let $\muh_0 \in \Prob_c (\Sigmah)$ be an ergodic measure w.r.t. $f$ and let $\phi \in C_b (X \times \T^d)$ be an observable that depends on a finite number of coordinates of $\om \in X$.
Given any $\ep > 0$, there are  $\delta = \delta (\ep, \muh_0, \phi) > 0$,  $\nbar = \nbar(\ep, \muh_0, \phi) \in \N$ and $c = c (\ep, \muh_0, \phi) > 0$ such that for all $\muh \in \Prob_c (\Sigmah)$ with $W_1 (\muh, \muh_0) < \delta$, for all $\theta \in \T^d$ and for all $n \ge \nbar$ we have
\begin{equation}\label{mixed base LDT}
\muh^\Z \left\{ \om \in X \colon \abs{ \frac{1}{n} \sum_{j=0}^{n-1} \phi (f^j (\om, \theta) ) - \int_{X \times \T^d} \phi  \, d(\muh^\Z\times m) }  \ge  \ep \right\}  <  e^{- c  n}  \, .
\end{equation}

A similar estimate also holds with the integral in~\eqref{mixed base LDT} taken with respect to the fixed measure $\muh_0^\Z\times m$.
\end{theorem}

\begin{proof}
We use a stopping time argument where the times are chosen uniformly in $\theta$; this way we decouple the variables $\om$ and $\theta$ and reduce the problem to a concentration inequality over the Bernoulli shift $\sigma$ for an observable that depends on finitely many random coordinates.

Fix $\ep>0$.
It is easy to see that if we established~\eqref{mixed base LDT} with some constant $a (\phi)$ instead of $ \int  \phi \,  d(\nu^\Z\times m)$, then we would have
$$\abs{a (\phi) -  \int_{X \times \T^d} \phi  \, d(\nu^\Z\times m)} < \ep + \norm{\phi}_0 \, e^{- c  n}  < 2 \ep $$
for $n$ large enough, which would therefore imply~\eqref{mixed base LDT} as written. We will then establish the estimate with $a (\phi) := \int \phi \, d (\muh_0^\Z \times m)$.

Note also that replacing $\phi$ by $- \phi$, it suffices to prove the upper bound in~\eqref{mixed base LDT}, namely that for $\om$ outside  an exponentially small set with respect to the $\nu^\Z$ measure, and for all $\theta\in\T^d$ we have
\begin{equation}\label{mixed base LDT2}
\frac{1}{n} \sum_{j=0}^{n-1} \phi (f^j (\om, \theta) ) <  a (\phi)  + \ep \, .
\end{equation}

Finally, replacing  $\phi$ by $\phi - \inf \phi$, we may assume  that $\phi \ge 0$.

Using Lemma~\ref{Birkhoff LDE1}, for $\muh_0^\Z$-a.e. $\om \in X$ we can define $n (\om) = n (\om, \ep)$ to be the first integer such that for all $\theta \in \T^d$,
$$
\frac{1}{n(\om)} \sum_{j=0}^{n(\om)-1} \phi (f^j (\om, \theta) ) <  a (\phi) + \ep \, .
$$

Given $m \in \N$, let
\begin{align*}
\U_m & := \{ \om \in X \colon n (\om) \le m \}  \\
& = \bigcup_{k=1}^m  \, \left\{ \om \in X \colon \frac{1}{k} \sum_{j=0}^{k-1} \phi \circ f^j (\om, \theta) ) <  a (\phi) + \ep \quad \forall \theta \in \T^d  \right\} \, .
\end{align*}

Since $\phi$ and $f$ are continuous and $\T^d$ is compact, the set $\U_m$ is open. 
Moreover, as the sequence of sets $\{\U_m\}_{m\ge1}$ increases to a full $\muh_0^Z$-measure set, there is $N = N (\ep, \muh_0, \phi)$ such that $\muh_0^\Z (\U_N\comp) < \ep$. 

Note that since the observable $\phi$ depends on a finite number (say $k_0$) of coordinates,  the set $\U_N$ is determined by  $k := k_0+N$ coordinates, where $k = k (\ep, \muh_0, \phi)$. The same of course holds for its complement $\U_N\comp$, which is a closed set. Let $L \subset \Sigmah^k$ be the projection of $\U_N\comp$ in the $k$ coordinates on which it depends. Then $L$ must be a closed set and for any $\muh \in \Prob_c (\Sigmah)$ we have $ \muh^\Z (\U_N\comp)= \muh^k (L) $. 

We claim that if a measure $\muh$ is chosen sufficiently close to $\muh_0$ relative to the Wasserstein distance, we can ensure that the $\muh^\Z$ measure of $\U_N\comp$ is of order $\ep$ as well. Indeed, applying Lemma~\ref{Urysohn} to the closed set $L\subset \Sigmah^k$, there are an open set $D \supset L$ such that 
$$\muh_0^k (D) \le \muh_0^k (L) + \ep =  \muh_0^\Z ( \U_N\comp) + \ep < 2 \ep $$  
and  a Lipschitz continuous function $g \colon \Sigmah^k \to [0, 1]$ such that $\ind_L \le g \le \ind_D$ and 
$\norm{g}_\Lip = C = C (\ep, L, k) = C (\ep, \muh_0, \phi)$.

It is easy to see that for any $\muh \in \Prob_c (\Sigmah)$ we have
$$W_1 (\muh^k, \muh_0^k) \le k \, W_1 (\muh, \muh_0) \, ,$$
so
$$
\abs{ \int_{\Sigmah^k} g \, d (\muh^k - \muh_0^k) }  \le C \, W_1 (\muh^k, \muh_0^k) \le C k  \, W_1 (\muh, \muh_0) .
$$

Then
\begin{align*}
\nu^\Z (\U_N\comp) & = \nu^k ( L) = \int_{\Sigmah^k} \ind_L \, d \muh^k \le  \int_{\Sigmah^k} g \, d \muh^k
\le  \int_{\Sigmah^k} g \, d \muh_0^k + C k \, W_1 (\muh, \muh_0)\\
& \le \int_{\Sigmah^k} \ind_D \, d \muh_0^k +  C k\, W_1 (\muh, \muh_0) = \muh_0^k (D) + C k \, W_1 (\muh, \muh_0) < 3 \ep \, ,
\end{align*}
provided that $W_1 (\muh, \muh_0) < \delta =: \frac{\ep}{C k}$.

\smallskip

By design, for all $\om \in \U_N$ we have $1 \le n (\om) \le N$ and for all $\theta \in \T^d$,
\begin{equation}\label{base eq1}
 \sum_{j=0}^{n(\om)-1} \phi (f^j (\om, \theta) ) \le n (\om) \, a (\phi)  + n (\om) \, \ep \, .
 \end{equation}

Fix any $\om = \{\om_j\}_{j\in\Z} \in X$ and define inductively
a sequence of integers $\{n_k = n_k (\om)\}_{k\ge1}$  as follows.

If $\om \in \U_N$ then $n_1 := n (\om)$, otherwise $n_1 :=1$.

If $\sigma^{n_1} \om \in \U_N$ then $n_2 := n (\sigma^{n_1} \om)$, otherwise $n_2 :=1$.

If, for $k\ge1$, we have $\sigma^{n_k+\ldots+n_1} \om \in \U_N$ then $n_{k+1} := n (\sigma^{n_k+\ldots+n_1} \om)$, otherwise $n_{k+1} := 1$. Note that $1\le n_k \le N$ for all $k\ge1$.

Using~\eqref{base eq1} (and the fact that $\phi \ge 0$), for all $\theta \in \T^d$, the Birkhoff sum of length $n_1$ with starting phase  $(\om, \theta)$ has the bound
$$ \sum_{j=0}^{n_1-1} \phi (f^j (\om, \theta) ) \le n_1  a (\phi)  + n_1 \, \ep  +  \norm{\phi}_{0} \ind_{\U_N\comp} (\om)  \, .$$

Similarly, the Birkhoff sum of length $n_2$ with starting phase $f^{n_1} (\om, \theta) = (\sigma^{n_1} \om, \theta + \alfa (\om_0) + \ldots + \alfa(\om_{n_1-1}))$ has the bound
$$ \sum_{j=0}^{n_2-1} \phi (f^{j+n_1} (\om, \theta) ) \le n_2 a (\phi)    + n_2 \, \ep +  \norm{\phi}_{0} \ind_{\U_N\comp} (\sigma^{n_1} \om)  \, .$$

In general, for $k\ge 1$, the Birkhoff sum of length $n_{k+1}$ with starting phase $f^{n_k+\ldots + n_1} (\om, \theta) = (\sigma^{n_k+\ldots+n_1} \om, \theta + \alfa(\om_0) + \ldots +  \alfa (\om_{n_k+\ldots+n_1 -1}) )$ has the bound
$$ \sum_{j=0}^{n_{k+1}-1} \phi (f^{j+n_k + \ldots + n_1} (\om, \theta) ) \le n_{k+1} a (\phi)   \,  + \, n_{k+1}  \ep \, + \,   \norm{\phi}_{0} \ind_{\U_N\comp} (T^{n_k+\ldots+n_1} \om)   .$$

Let $\nbar = \nbar (\ep, \mu, \phi) :=  N \max \left\{ \frac{\norm{\phi}_0}{\ep}, 1\right\}$, so $\nbar \ge N \ge n_1$. Fix any $n \ge \nbar$. Since $n_1 < n_1 + n_2 < \ldots < n_1+\ldots + n_k < \ldots  $ , there is $p \ge 1$ such that
$n =  n_1 + \ldots n_{p} + m$, where $0 \le m < n_{p+1} \le N$.

It follows that
\begin{align*}
\sum_{j=0}^{n-1} \phi (f^j (\om, \theta) ) & =   \sum_{j=0}^{n_1+\ldots+n_p-1} \phi (f^j (\om, \theta) ) +
 \sum_{j=0}^{m-1} \phi (f^{j+n_1+\ldots+n_p} (\om, \theta) ) \\
 & =   \sum_{k=0}^{p-1}  \sum_{j=0}^{n_{k+1}-1}\phi (f^{j+n_k+\ldots+n_1} (\om, \theta) ) \\
 & \kern1em +
 \sum_{j=0}^{m-1} \phi (f^{j+n_1+\ldots+n_p} (\om, \theta) ) \, ,
 \end{align*}

 hence
 \begin{align*}
\sum_{j=0}^{n-1} \phi (f^j (\om, \theta) )
 & \le (n_1+\ldots + n_p) a (\phi) + (n_1+\ldots + n_p) \ep \\
 &  \kern1em + \norm{\phi}_{0} \sum_{k=0}^{p-1} \ind_{\U_N\comp} (\sigma^{n_k+\ldots+n_1} \om)
 + m \norm{\phi}_{0} \\
 & \le n a (\phi)+ n \ep + \sum_{j=0}^{n-1} \ind_{\U_N\comp} (\sigma^j \om)  + N \norm{\phi}_{0} \\
 & <   n a (\phi) + 2 n \ep + \sum_{j=0}^{n-1} \ind_{\U_N\comp} (\sigma^j \om) \, .
\end{align*}

We obtained the following: for all $\om \in X$, $\theta \in \T^d$ and $n\ge\nbar$,
\begin{equation}\label{split qp random}
\frac{1}{n} \sum_{j=0}^{n-1} \phi (f^j (\om, \theta) ) < a (\phi) + 2 \ep + \frac{1}{n} \sum_{j=0}^{n-1} \ind_{\U_N\comp} (\sigma^j \om) \, .
\end{equation}

It remains to estimate the Birkhoff average over the Bernoulli shift of the function $\ind_{\U_n\comp}$. Since $\ind_{\U_n\comp}$ depends on $k$ coordinates, its $n$-th Birkhoff average depends on $n+k-1$ coordinates. Hence the following function is well defined (and it is measurable):
$$h \colon \Sigmah^{n+k-1} \to \R, \quad h (x_0, \ldots, x_{n+k-2}) := \frac{1}{n} \sum_{j=0}^{n-1} \ind_{\U_N\comp} (\sigma^j \om) \, ,$$
where $\om = \{\om\}_{j\in \Z} \in X$ with $\om_0 = x_0, \ldots, \om_{n-k-2} = x_{n-k-2}$.

Because of the dependence of $\ind_{\U_N\comp}$ on $k$ coordinates, the function $h$ satisfies the following bounded differences property:
\begin{align*}
\left| h (x_0, \ldots, x_{i-1}, x_i, x_{i+1}, \ldots, x_{n+k-2}) \right. & \\
\left. - h (x_0, \ldots, x_{i-1}, x_i', x_{i+1}, \ldots, x_{n+k-2}) \right| &
\le \frac{2 k}{n} \norm{h}_\infty = \frac{2 k}{n} \, .
\end{align*}

Then by McDiarmid's inequality (see~\cite[Theorem 3.1]{McD}), for any probability measure $\muh$ on $\Sigmah$ there is an exceptional set $\B_n \subset \Sigmah^{n+k-1}$ with $\muh^{n+k-1} (\B_n) < e^{-\frac{\ep^2}{2 k^2} n} $, so that for $(\om_0, \ldots, \om_{n+k-2}) \notin \B_n$  we have:
$$ h  (\om_0, \ldots, \om_{n+k-2}) - \int h \, d \nu^{n+k-1} <  \ep  \, .$$

Clearly
\begin{align*}
& \int h (\om_0, \ldots, \om_{n+k-2}) \, d \muh^{n+k-1} (\om_0, \ldots, \om_{n+k-2})   \\
 = & \int \frac{1}{n} \sum_{j=0}^{n-1} \ind_{\U_N\comp} (\sigma^j \om) \, d \muh^\Z (\om)  \\
 = & \int \ind_{\U_N\comp} (\om) \, d \muh^\Z (\om)
= \muh^\Z (\U_N\comp) < 3 \ep \, ,
\end{align*}
which when combined with~\eqref{split qp random} implies~\eqref{mixed base LDT2}.
\end{proof}

%% file: fiberdyn.tex
In this section we formally introduce the concept of mixed random-quasiperiodic cocycle, present a motivating example and study the upper semicontinuity of its maximal Lyapunov exponent. 

\subsection{The group of quasiperiodic cocycles}
A quasiperiodic cocycle is a skew-product map of the form 
$$\T^d \times \R^m \ni (\theta, v) \mapsto \left( \tau_\alpha (\theta),  A(\theta) v \right) \in \T^d \times \R^m \, ,$$
where $\tau_\alpha (\theta) = \theta + \alpha$ is a translation on $\T^d$ by a rationally independent frequency $\alpha \in \T^d$ and $A \in C^0 (\T^d, \SL_m (\R))$ is a continuous matrix valued function on the torus.

This cocycle can thus be identified with the pair $(\alpha, A)$. Consider the set
$$\Gscr=\Gscr(d,m):= \T^d\times C^0(\T^d,\SL_m(\R))$$
of all quasiperiodic cocycles.  

This set is a Polish metric space when equipped with the product metric (in the second component we consider the uniform distance). The space $\Gscr$ is also a group, and in fact a topological group, with the natural composition and inversion operations
\begin{align*}
(\alpha,A)\circ (\beta, B) &:= (\alpha+\beta, (A\circ \tau_\beta) \, B)  \\
(\alpha,A)^{-1} &:= ( -\alpha, (A\circ \tau_{-\alpha})^{-1} ) \, .
\end{align*}

Given $\muh \in \Prob_c (\Gscr)$ let $\om = \left\{ \om_n \right\}_{n\in\Z}$, $\om_n = (\alpha_n, A_n)$ be an i.i.d. sequence of random variables  in $\Gscr$  with law $\muh$. Consider the corresponding multiplicative process in the group $\Gscr$
\begin{align*}
\Pi_n & = \om_{n-1} \circ \ldots \circ \om_1 \circ \om_0 \\
& = \left( \alpha_{n-1} + \ldots + \alpha_1 + \alpha_0, \, ( A_{n-1} \circ \tau_{\alpha_{n-2} + \ldots + \alpha_0}) \ldots (A_1 \circ \tau_{\alpha_0} ) \, A_0   \right) \, .
\end{align*}

In order to study this process in the framework of ergodic theory, we model it by the iterates of a  linear cocycle.

 \subsection{Mixed random-quasiperiodic cocycles} Given $\muh \in \Prob_c (\Gscr)$, let $\Sigmah \subset \Gscr$ be a closed subset (thus a a Polish space as well) such that $\Sigmah \supset \supp \muh$. Depending on what will be convenient in a specific situation, $\Sigmah$ can be the entire space $\Gscr$, or a compact set, say $\supp \muh$ or, for a given constant $L < \infty$, the set $\Gscr_L := \{ (\alpha, A) \in \Gscr \colon \norm{A}_0 \le L \}$.
In any case, the choice of the set $\Sigmah \supset \supp \muh$ will not influence the definitions and results to follow.

We regard $(\Sigmah, \muh)$ as a space of symbols and consider, as before, the shift $\sigma$ on the space $X := \Sigmah^\Z$ of sequences $\om = \left\{ \om_n \right\}_{n\in\Z}$ endowed with the product measure $\muh^Z$ and the product topology (which is metrizable). 
The standard projections 
\begin{align*}
\alfa \colon \Sigmah \to \T^d, & \qquad \alfa (\alpha, A) = \alpha \\
\Ascr \colon \Sigmah \to C^0 (\T^d, \SL_m (\R)), & \quad \quad \Ascr (\alpha, A) = A
\end{align*}
determine the linear cocycle $F=F_{(\alfa,\Ascr)} \colon X\times\T^d\times\R^m \to X\times\T^d\times\R^m$ defined by
$$ F(\omega,\theta, v) := \left(\sigma \omega, \theta+\alfa(\omega_0), \Ascr(\omega_0)(\theta)\, v \right)  .$$
The non-invertible version of this map, with the same expression, is defined on $\Xplus\times \T^d\times \R^m$, where $\Xplus=\Sigmah^{\N}$.

Thus the base dynamics of the cocycle $F$ is the mixed random-quasiperiodic map 
$$ X \times \T^d \ni (\om, \theta) \mapsto \left(\sigma \om, \theta + \alfa (\om_0) \right) \in X \times \T^d ,$$
while the fiber action is induced by the map
$$X \times \T^d \ni ( \om, \theta ) \mapsto \Ascr (\om, \theta) =: \Ascr (\om_0) (\theta) \in \SL_m (\R) .$$
The skew-product $F$ will then be referred to as a mixed random-quasiperiodic cocycle.

For $\om = \{ \om_n \}_{n\in\Z} \in X$ and $j \in \N$ consider the composition of random translations
\begin{align*}
\tau_\om^j & := \tau_{\alfa(\om_{j-1})} \circ \ldots \circ \tau_{\alfa(\om_0)} 
 = \tau_{ \alfa(\om_{j-1}) + \ldots +  \alfa(\om_0) } = \tau_{\alfa (   \om_{j-1} \circ \ldots \circ \, \om_0  ) } \, .
\end{align*}

The iterates of the cocycle $F$ are then given by
$$F^n (\om, \theta, v) = \left( \sigma^n \om, \tau_\om^n (\theta), \, \Ascr^n (\om) (\theta) v \right) ,$$
where
\begin{align*}
\Ascr^n (\om) & = \Ascr \left( \om_{n-1} \circ \ldots \circ \om_1 \circ \om_0 \right) \\
& = \left( \Ascr (\om_{n-1}) \circ \tau_\om^{n-2}  \right) \, \ldots \,  \left( \Ascr (\om_{1}) \circ \tau_\om^{0}  \right) \, \Ascr (\om_0) \, .
\end{align*}
Thus  $\Ascr^n(\omega)$ can be interpreted as a random product of
quasiperiodic cocycles. For convenience we also denote $\Ascr^n (\om, \theta) := \Ascr^n (\om) (\theta)$.

\medskip

By the subadditive ergodic theorem, the limit of $\displaystyle \frac{1}{n} \, \log \norm{ \Ascr^n (\om) (\theta)}$ as $n\to\infty$ exists for $\muh^\Z \times m$ a.e. $(\om, \theta) \in X \times \T^d$. If the base dynamics $f$ is ergodic w.r.t. 
$\muh^\Z \times m$, then this limit is a constant that depends only on the measure $\muh$ and it is called the maximal Lyapunov exponent of the cocycle $F$, which we denote by $L_1 (\muh)$. Thus
\begin{align*}
L_1 (\muh) & = \lim_{n\to\infty} \frac{1}{n} \, \log \norm{ \Ascr^n (\om) (\theta)} \quad \text{for} \quad \muh^\Z \times m \text{ a.e. } (\om, \theta) \\
& = \lim_{n\to\infty} \, \int_{X \times \T^d}  \frac{1}{n} \, \log \norm{ \Ascr^n (\om) (\theta)} \, d ( \muh^\Z \times m ) \, .
\end{align*}

An important problem, to be studied more in depth in future projects concerns the continuity properties of the map $ \muh \mapsto L_1 (\muh)$.

\begin{remark}\label{alt def}
An alternative, somewhat more particular way to define mixed random-quasiperiodic cocycles is the following. Fix an abstract space of symbols $(\Sigmah, \rho)$ (where $\Sigmah$ is a Polish metric space and $\rho \in \Prob_c (\Sigmah)$) and a continuous function $\alfa \colon \Sigmah \to \T^d$. Consider the corresponding mixed quasiperiodic base dynamics $(X \times \T^d, f, \rho^\Z \times m)$ defined in Section~\ref{base}, where $X = \Sigmah^\Z$. A continuous function $\Ascr \colon X \times \T^d \to \SL_m (\R)$ that depends only the coordinate $\om_0$ of $\om \in X$ and on $\theta \in \T^d$ determines the linear cocycle $F = F_{(\alfa, \Ascr)}$ over $f$ given by 
\begin{equation} \label{def 111}
F (\om, \theta, v) = \left( f (\om, \theta), \, \Ascr (\om, \theta) v \right) = 
\left( \sigma \om, \theta + \alfa (\om_0), \Ascr (\om, \theta) v \right) \, .
\end{equation}

Note that since $\Ascr$ depends only on the coordinate $\om_0$  of $\om \in X$ and on $\theta \in \T^d$, it can be identified with the map
$$\Sigmah \ni \om_0 \mapsto \Ascr (\om_0) \in C^0 (\T^d, \SL_m (\R)) , \quad \Ascr (\om_0) (\theta) =  \Ascr (\om, \theta)
\, .$$

Then setting 
$$\muh := \alfa_\ast \rho \times \Ascr_\ast \rho \in \Prob_c (\Gscr)  , $$
we conclude that the cocycle $F_{(\alfa, \Ascr)}$ defined in~\eqref{def 111} can also be realized as a cocycle driven by a measure, namely the push forward measure $\muh \in \Prob_c (\Gscr)$ above.

The space of mixed cocycles $F_{(\alfa, \Ascr)}$ is a metric space with the uniform distance
$$
\dist \left( (\alfa,\Ascr), \, (\alfa',\Ascr') \right)= \norm{\alfa-\alfa'}_0 + \norm{\Ascr-\Ascr'}_0 \, .
$$

Note that the map $ (\alfa,\Ascr) \mapsto \alfa_\ast \rho \times \Ascr_\ast \rho \in \Prob_c (\Gscr)$ is Lipschitz continuous 
(recall that $\Prob_c (\Gscr)$ is equipped with the Wasserstein distance). 
\end{remark}

\subsection{Upper semicontinuity of the Lyapunov exponent}
We derive a nearly uniform upper semicontinuity of the Lyapunov exponent of a mixed cocycle, a technical result in the spirit of~\cite[Proposition 3.1]{DK-book}. This is a type of uniform {\em upper} large deviations estimate, to be employed in future related projects. For now, as a consequence of this estimate, we establish the upper semicontinuity of the Lyapunov exponent as a function of the measure, relative to the Wasserstein distance. 

Fix a number $L < \infty$, let $\Sigmah := \Gscr_L$,  $ X = \Sigmah^\Z$ and consider the mixed random-quasiperiodic dynamics $f$ on $X \times \T^d$.

\begin{theorem}\label{fiber upper ldt}
Let $\muh_0 \in \Prob_c (\Sigmah)$ be an ergodic measure w.r.t. $f$.
Given any $\ep > 0$, there are  $\delta = \delta (\ep, \muh_0, L) > 0$,  $\nbar = \nbar(\ep, \muh_0, L) \in \N$ and $c = c (\ep, \muh_0, L) > 0$ such that for all $\muh \in \Prob_c (\Sigmah)$ with $W_1 (\muh, \muh_0) < \delta$, for all $\theta \in \T^d$ and for all $n \ge \nbar$ we have
\begin{equation}\label{eq fiber upper ldt}
\muh^\Z \left\{ \om \in X \colon  \frac{1}{n} \log \norm{\Ascr^n (\om)  (\theta)} \ge L_1 (\muh_0) + \ep \right\}  <  e^{- c  n}  \, .
\end{equation}

Moreover, the map $\muh \mapsto L_1 (\muh)$ is upper semicontinuous with respect to the Wasserstein metric in the space of ergodic measures. 

\end{theorem}

\begin{proof} The argument is similar to the one used in the proof of Theorem~\ref{base ldt}.  Let
$$a_n (\om, \theta) := \log \norm{ \Ascr^n (\om) (\theta)} $$
and note that the sequence $\{a_n\}_{n\ge1}$ is $f$-subadditive, that is, for all $n, m \in \N$ and $(\om, \theta) \in X\times \T^d$  we have:
$$a_{n+m} (\om, \theta) \le a_n (\om, \theta) + a_m (f^m (\om, \theta)) \, .$$

For $(\om, \theta) \in X \times \T^d$ let $n (\om, \theta)$ be the least positive integer $n$ such that
\begin{equation}\label{eq 100}
\frac{1}{n} a_n (\om, \theta)  <  L_1 (\muh_0) + \ep \, .
\end{equation}

By Kingman's ergodic theorem, $n (\om, \theta)$ is defined for $\muh_0^\Z \times m$-a.e. $(\om, \theta)$ and, moreover, for $m \in \N$, if we denote by
$$\U_m := \left\{ (\om, \theta) \colon n (\om, \theta) \le m\right\} \, ,$$
it follows that $\U_m$ increases to a full $\muh_0^\Z \times m$-measure set as $m\to\infty$.
Then there is $N = N (\ep, \muh_0)$ such that $\muh_0^\Z \times m (\U_N\comp) < \ep$.

We note that a priori we do not have an exact analogue of Lemma~\ref{Birkhoff LDE1}, that is, the uniformity in $\theta$ of the convergence in Kingman's ergodic theorem.\footnote{A posteriori our result provides the {\em upper}  uniformity in $\theta$. We note that a lower uniformity result, and hence uniform convergence in $\theta$ in Kingman's theorem is in general not possible, see also~\cite{Furman1}.}  We perform a stopping time argument corresponding to the behavior of the $f$-orbit of a point $(\om, \theta)$; using the subadditivity of the sequence $\{a_n\}_{n\ge1}$, we eventually reduce the problem to the additive situation in Theorem~\ref{base ldt}.

Let $C = C (L, \muh_0) := \sup \left\{ \log \norm{ \Ascr (\om_0) (\theta)} \colon \om_0 \in \Sigmah, \theta \in \T^d   \right\} < \infty$.

Fix an arbitrary point $(\om, \theta) \in X \times \T^d$ and define inductively the sequence of stopping times $\{n_k = n_k (\om, \theta)\}_{k\ge1}$ as follows.

If $(\om, \theta) \in \U_N$, let $n_1 := n (\om, \theta)$, otherwise $n_1 :=1$.

If $f^{n_1} (\om, \theta) \in \U_N$, let $n_2 := n (f^{n_1} (\om, \theta))$, otherwise $n_2 :=1$.

For $k\ge1$, if $f^{n_k+\ldots+n_1} (\om, \theta) \in \U_N$ then $n_{k+1} := n (f^{n_k+\ldots+n_1} (\om, \theta))$, otherwise $n_{k+1} := 1$. Note that for all $k\ge1$ we have $1\le n_k \le N$ and by~\eqref{eq 100},
\begin{equation*}\label{eq 101}
a_{n_k} (f^{n_1+\ldots+n_{k-1}} (\om, \theta)) \le n_k (L_1 (\muh) + \ep) + C \, \ind_{\U_N\comp} ( f^{n_1+\ldots+n_{k-1}} (\om, \theta) ) \, .
\end{equation*}

Let $\nbar = \nbar (\ep, \muh, \Sigmah) :=  N \max \left\{ \frac{C}{\ep}, 1\right\}$, so $\nbar \ge N \ge n_1$. Fix any $n \ge \nbar$. Since $n_1 < n_1 + n_2 < \ldots < n_1+\ldots + n_k < \ldots $, there is $p \ge 1$ such that
$n =  n_1 + \ldots n_{p} + m$, where $0 \le m < n_{p+1} \le N$.

Using the subadditivity of the sequence $\{a_n\}_{n\ge1}$ it  follows that
\begin{align*}
a_n (\om, \theta)   \le  & \ a_{n_1} (\om, \theta) + a_{n_2} (f^{n_1} (\om, \theta)) + \ldots + a_{n_p} (f^{n_1+\ldots+n_{p-1}} (\om, \theta)) \\
& + a_m (f^{n_1+\ldots+n_{p}} (\om, \theta)) \\
\le & \ (n_1+\ldots n_p) \, (L_1 (\muh) + \ep) + C \sum_{j=0}^{n-1}   \ind_{\U_N\comp} ( f^j (\om, \theta) ) + C N \, .
\end{align*}

Hence for all $(\om, \theta) \in \Sigmah^\Z \times \T^d$ and for all $n \ge \nbar$ we have
$$\frac{1}{n} \, \log \norm{ \Ascr^n (\om) (\theta)}  \le L_1 (\muh) + 2 \ep + C \, \frac{1}{n} \,  \sum_{j=0}^{n-1}   \ind_{\U_N\comp} ( f^j (\om, \theta) ) \, .$$

The closed set $\U_n\comp \subset \Sigmah^\Z \times \T^d$ is determined by the coordinates $\om_0, \ldots, \om_{N-1}$ and $\theta$. Therefore, as in the proof of Theorem~\ref{base ldt}, using Lemma~\ref{Urysohn}, there are an open set $D \supset \U_N$ with $(\muh_0^\Z \times m) \,  (D) < 2 \ep$ and a Lipschitz continuous function $g \colon X \times \T^d \to [0, 1]$
which depend only on the coordinates $\om_0, \ldots, \om_{N-1}, \theta$  such that $\ind_{\U_N\comp} \le \phi \le \ind_D$.

Thus for all $(\om, \theta) \in X\times \T^d$ and $n \ge \nbar$ we have:
$$\frac{1}{n} \,  \sum_{j=0}^{n-1}   \ind_{\U_N\comp} ( f^j (\om, \theta) ) \le
\frac{1}{n} \,  \sum_{j=0}^{n-1}  g ( f^j (\om, \theta) ) \, .$$

The observable $g$ depends of course only on $\ep$ and $\muh_0$. Applying Theorem~\ref{base ldt} to $g$, for any measure $\muh$ on $\Sigmah$ that is sufficiently close (depending on $\ep, \muh_0$) to $\muh_0$ in the Wasserstein distance and for all $\theta \in \T^d$ we have:
\begin{equation*}\label{eq 111}
 \frac{1}{n} \sum_{j=0}^{n-1} g (f^j (\om, \theta) ) <  \int g  \, d(\muh_0^\Z\times m)   +  \ep
 \end{equation*}
for $\om$ outside a set of $\muh^\Z$-measure $< e^{- c  n}$, where $c = c (\ep, \muh_0) > 0$.

Moreover, 
\begin{align*}
\int g  \, d(\muh_0^\Z\times m) & \le \int \ind_D  \, d(\muh_0^\Z\times m) = (\muh_0^\Z\times m) (D) < 2 \ep ,
\end{align*}
which combined with the previous estimates proves~\eqref{eq fiber upper ldt}.

Finally, using the estimate~\eqref{eq fiber upper ldt} and integrating with respect to the measure $\muh^\Z\times m$, where $\muh$ is close enough to $\muh_0$ in the Wasserstein metric, for all large enough $n$ we have
\begin{align*}
\int  \frac{1}{n} \log \norm{\Ascr^n (\om, \theta)} \, d (\muh^\Z\times m)  & \le L_1 (\muh_0) + \ep + L \, e^{- c n} 
\\
& < L_1 (\muh_0) + 2 \ep \, .
\end{align*}

Restricting to measures $\muh$ that are ergodic with respect to $f$ and letting $n \to \infty$ we conclude that $L_1 (\muh) < L_1 (\muh_0) + 2 \ep$.
 \end{proof}

\begin{remark}
A uniform {\em lower} large deviations estimate (and hence a full, uniform large deviations type estimate) that is, a bound like
$$
\muh^\Z \left\{ \om \in X \colon  \frac{1}{n} \log \norm{\Ascr^n (\om)  (\theta)} \le L_1 (\muh_0) - \ep \right\}  <  e^{- c  n}  \, .
$$
cannot hold at this level of generality. 

If it did, then (at least restricting to cocycles defined as in Remark~\ref{alt def}), by the Abstract Continuity Theorem, (see~\cite[Theorem 3.1]{DK-book}) we would conclude that the Lyapunov exponent is a continuous function. However, this is not necessarily the case without stronger assumptions on the data.  

Indeed, let $\muh=\frac{1}{2}\delta_{(0,I)}+\frac{1}{2}\delta_{(\alpha, A)}$ where $(\alpha, A)$ is the quasiperiodic cocycle constructed in~\cite{Wang-You} and shown to be a point of discontinuity of the Lyapunov exponent. Then $L_1(\muh)=\frac{1}{2}L_1(\alpha, A)>0$. However, $\muh$ can be approximated by measures $\muh_n$ with zero Lyapunov exponent, e.g. $\muh_n=\frac{1}{2}\delta_{(0,I)}+\frac{1}{2}\delta_{(\alpha, A_n)}$, where $\{(\alpha,A_n)\}_{n\ge1}$  is the approximating sequence of $(\alpha, A)$ in~\cite{Wang-You}. One may consult~\cite[Section 5]{BPS} for more details.
\end{remark}

%% file: roadmap.tex

The study of linear cocycles in general and of mixed cocycles in particular is motivated in part by their relationship with discrete Schr\"odinger operators (see~\cite{Damanik-survey} for a review of this topic).

Recall the discrete quasiperiodic Schr\"odinger operator given by
 \begin{equation}\label{road 0}
(H_{{\rm qp}} (\theta) \, \psi)_n=- \psi_{n+1} - \psi_{n-1}+ v (\theta+n\alpha) \, \psi_{n}, \quad \forall n\in \Z,
\end{equation}
for some potential function $v \in C^0 (\T^d, \R)$ and ergodic frequency $\alpha\in\T^d$. Given an energy $E\in\R$, consider the corresponding Schr\"odinger cocycle $(\alpha, S_E)$, where $S_E \in C^0 (\T^d, \SL_2 (\R))$,
$$
S_E (\theta)=\begin{pmatrix} v(\theta)-E & -1 \\ 1 & 0 \end{pmatrix} .
$$
Let $\Gscr = \T^d \times C^0 (\T^d, \SL_2 (\R))$ be the space of $\SL_2 (\R)$ valued quasiperiodic cocycles.

We will describe different types of random perturbations of the operator~\eqref{road 0} and the associated mixed Schr\"odinger cocycle.  

Given $\rho \in \Prob_c (\R)$, consider an i.i.d. sequence of random variables $\{w_n\}_{n\in\Z}$ with common law $\rho$. Interpreting $\{w_n\}_{n\in\Z}$ as random perturbations of the quasiperiodic potential $v_n (\theta) = v (\theta + n \alpha)$, we obtain the Schr\"odinger operator
 \begin{equation}\label{road 1}
(H  \, \psi)_n=- \psi_{n+1} - \psi_{n-1} + \left( v (\theta+n\alpha) + w_n \right)  \, \psi_{n}, \quad \forall n\in \Z .
\end{equation} 

Note that putting $P (\om) =\begin{pmatrix} 1 & \om \\ 0 & 1 \end{pmatrix}$, we can write
$$\begin{pmatrix} v(\theta) + \om -E & -1 \\ 1 & 0 \end{pmatrix} =
\begin{pmatrix} 1 & \om \\ 0 & 1 \end{pmatrix} \, \begin{pmatrix} v(\theta)  -E & -1 \\ 1 & 0 \end{pmatrix} 
= 
P(\om) \, S_E (\theta) \, .$$

The Schr\"odinger cocycle associated to the operator~\eqref{road 1} is then the mixed random-quasiperiodic cocycle driven by the measure $\muh_E \in \Prob_c (\Gscr)$ given by
\begin{equation}\label{road 11}
\muh_E = \delta_\alpha \times \int_\R  \delta_{P (\om) S_E} \, d \rho (\om) \, .
\end{equation}

\medskip

A very different model is obtained if instead we randomize the translation by the frequency  $\alpha$. Given $\mu \in \Prob (\T^d)$, let $\{\alpha_n\}_{n\in\Z}$ be an i.i.d. sequence of random variables with common law $\mu$ and consider the Schr\"odinger operator
\begin{equation}\label{road 2}
(H (\theta) \, \psi)_n=- \psi_{n+1} - \psi_{n-1} +  v (\theta+\alpha_0 + \ldots + \alpha_{n-1})  \, \psi_{n}, \quad \forall n\in \Z .
\end{equation} 

The Schr\"odinger cocycle associated to the operator~\eqref{road 2} is the mixed random-quasiperiodic cocycle driven by the measure $\muh_E \in \Prob_c (\Gscr)$ given by
\begin{equation}\label{road 22}
\muh_E = \mu \times  \delta_{S_E}  \, .
\end{equation}
 
\medskip

We may of course randomize both the frequency and the potential, by considering
\begin{equation*}\label{road 3}
(H (\theta) \, \psi)_n=- \psi_{n+1} - \psi_{n-1} +  \left( v (\theta+\alpha_0 + \ldots + \alpha_{n-1}) + w_n \right)  \, \psi_{n}, \ \forall n\in \Z .
\end{equation*} 
The corresponding cocycle is driven by
\begin{equation*}\label{road 33}
\muh_E = \mu \times  \int_\R  \delta_{P (\om) S_E} \, d \rho (\om) \, .
\end{equation*}

\medskip

As mentioned before, one of our  goals is to study the stability of the Lyapunov exponent of a quasiperiodic cocycle under random noise (with appropriate assumptions on the randomness), for Schr\"odinger or more general cocycles. 
To this end, in forthcoming papers we will consider an in depth study of these types of cocycles, as summarized below. 

\medskip

Firstly, we develop results of Furstenberg's theory on products of random matrices for our mixed random-quasiperiodic multiplicative processes. In particular we obtain a Furstenberg-type formula and generic criteria for the continuity as well as the positivity of the maximal Lyapunov exponent. Under general, easily checkable conditions, these criteria are applicable to the  mixed Schr\"odinger cocycles~\eqref{road 11} and~\eqref{road 22} thus establishing the continuity and the positivity of the Lyapunov exponents for all energies $E$. The latter property suggests that in some sense the randomness dominates the quasi-periodicity (under generic assumptions, random multiplicative processes have positive Lyapunov exponents, which is not  always the case for quasiperiodic ones). Furthermore, it will be interesting to see if, as with the Anderson model, the randomness in the  operator~\eqref{road 1} always leads to Anderson localization. This problem will be considered in the future. 

\medskip

The continuity of the Lyapunov exponent mentioned above is not effective, it is only a qualitative result. We will establish the H\"older continuity of the Lyapunov exponent of the Schr\"odinger cocycle~\eqref{road 1}, and in fact for cocycles driven by $\muh = \delta_\alpha \times \rho$, where $\rho$ is a measure on $C^0 (\T^d, \SL_m (\R))$. This is obtained via an abstract continuity theorem (ACT) (see~\cite[Chapter 3]{DK-book}) which depends on the availability of some uniform large deviations type (LDT) estimates on the iterates of the cocycle. The main goal of this future work is deriving such estimates.

We remark that the same problems for cocycles with random frequencies such as~\eqref{road 22}, even under stronger regularity assumptions, so far proved more intractable. 

\medskip
 
With the above LDT estimate for such mixed random-quasiperiodic cocycles at hand, we will then be able to let the amount of randomness tend to zero. More precisely, we will establish the stability under random noise of the LDT estimates for quasiperiodic cocycles. Combined with the ACT, this will prove the stability (i.e., in this case, continuity) of the Lyapunov exponent of quasiperiodic cocycles under random perturbations of the cocycle.

\medskip

Finally, another project will be dedicated to deriving statistical properties for the base mixed random-quasiperiodic dynamics (e.g. large deviations for more general observables and a central limit theorem).